\documentclass[10pt,twoside]{amsart}
\usepackage{amssymb}

\usepackage{amsmath, amsthm, amscd, amsfonts, amssymb, graphicx, color}
\usepackage[bookmarksnumbered, colorlinks, plainpages]{hyperref}

\newtheorem{thm}{Theorem}[section]

\newtheorem{lem}[thm]{Lemma}
\newtheorem{prob}[thm]{Problem}

\newtheorem{prop}[thm]{Proposition}
\newtheorem{defn}[thm]{Definition}

\numberwithin{equation}{section}

\begin{document}



\title{Total dominator chromatic number and \\
Mycieleskian graphs}
\author{Adel P. Kazemi
}

\thanks{{\scriptsize
\hskip -0.4 true cm MSC(2010): 
05C15; 05C69.
}}

\maketitle

\begin{abstract}
A total dominator coloring of a graph $G$ is a proper
coloring of $G$ in which each vertex of the graph is adjacent to every
vertex of some color class. The total dominator chromatic
number $\chi_d^t(G)$ of $G$ is the minimum number of color classes in a
total dominator coloring of it. In [Total dominator chromatic numer in graphs, submitted]
the author initialed to study this number in graphs and obtained some important results. Here,
we continue it in Mycieleskian graphs. We show that the total dominator chromatic
number of the Mycieleskian of a graph $G$ belongs to between $\chi_d^t(G)+1$ and $\chi_d^t(G)+2$,
and then characterize the family of graphs the their total dominator chromatic numbers are each of them.
\end{abstract}
\vskip 0.2 true cm


\pagestyle{myheadings}
\markboth{\rightline {\scriptsize  A. P. Kazemi}}
         {\leftline{\scriptsize Total dominator chromatic number and Mycieleskian graphs}}

\bigskip
\bigskip


\section{\bf Introduction}
\vskip 0.4 true cm

All graphs considered here are finite, undirected and simple. For
standard graph theory terminology not given here we refer to \cite{West}.
Let $G=(V,E) $ be a graph with the \emph{vertex set} $V$ of \emph{order}
$n(G)$ and the \emph{edge set} $E$ of \emph{size} $m(G)$. The
\emph{open neighborhood} of a
vertex $v\in V$ is $N_{G}(v)=\{u\in V\ |\ uv\in E\}$, while its cardinality is the \emph{degree} of $v$. The
\emph{minimum} and \emph{maximum degree} of $G$ are denoted by
$\delta =\delta (G)$ and $\Delta =\Delta (G)$, respectively. We write $K_n$, $C_{n}$ and $P_{n}$ for a \emph{complete graph} or a \emph{cycle} or a \emph{path} of order $n$, respectively, while $W_n$ denotes a \emph{wheel} of order $n+1$. The \emph{complement} of a graph $G$ is denoted by $\overline{G}$
and is a graph with the vertex set $V(G)$ and for every two vertices
$v$ and $w$, $vw\in E(\overline{G})$\ if and only if $vw\not\in
E(G)$.

Let $G=(V,E)$ be a graph with the vertex set $V=\{v_i\mid 1\leq i \leq n\}$.
The \emph{Mycieleskian graph} $M(G)$ of a graph $G$ is a graph with the vertex set
$V\cup U \cup \{w\}$ such that $U=\{u_i\mid 1\leq i \leq n\}$, and the edge
set $E\cup \{u_iv_j \mid v_iv_j\in E(G)\} \cup \{u_iw \mid u_i\in U\}$ \cite{West}.

A \emph{total dominating set} $S$ of a graph $G$ is a subset of the vertices in
it such that each vertex has at least one neighbor in $S$, that is, $N_G(v)\cap S\neq \emptyset$, and the
\emph{total domination number} $\gamma_t(G)$ of $G$ is the cardinality of a
minimum total dominating set \cite{hhs1}.

A \emph{proper coloring} of a graph $G$ is a function from
the vertices of the graph to a set of colors such that any two
adjacent vertices have different colors, and the \emph{chromatic number}
$\chi (G)$ of $G$ is the minimum number of colors needed in a proper
coloring of a graph \cite{West}.
In a proper coloring of a graph a \emph{color class} is the set of all same colored vertices of the graph.

In \cite{Kaz}, the author defined the new concept \emph{total dominator coloring in graphs} as following.
\begin{defn}
\label{total dominator coloring} \cite{Kaz} \emph{ A} total dominator coloring \emph{of a graph $G$, briefly TDC, is a proper
coloring of $G$ in which each vertex of the graph is adjacent to every
vertex of some color class. The }total dominator chromatic
number $\chi_d^t(G)$ \emph{of $G$ is the minimum number of color classes in a
TDC of $G$. A }$\chi_d^t(G)$-coloring \emph{of
$G$ is any total dominator coloring with $\chi_d^t(G)$ colors}.
\end{defn}

Here, we continue the studying of this number in Mycieleskian graphs. Exactly, we show that the total dominator chromatic
number of the Mycieleskian of a graph $G$ lies between $\chi_d^t(G)+1$ and $\chi_d^t(G)+2$,
and then characterize the family of graphs the their total dominator chromatic numbers are each of them.
First, we give some needed definitions, terminology and propositions from \cite{Kaz}.

If $f$ is a total dominator coloring or a proper coloring of $G$
with the coloring classes $V_1$, $V_2$, ..., $V_{\ell}$ such that every vertex in $V_i$ has color $i$, we write
simply $f=(V_1,V_2,...,V_{\ell})$. In the following two definitions $f=(V_1,V_2,...,V_{\ell})$ is a TDC of $G$.

\begin{defn}
\label{common neighborhood}
\emph{A vertex $v$ is called a }common neighbor \emph{of $V_i$ if $v\succ V_i$, that is, $v$ is adjacent to all vertices in $V_i$. The set of all common neighbors of $V_i$ is called the }common neighborhood \emph{of $V_i$ in $G$ and denoted by $CN_G(V_i)$.}
\end{defn}

\begin{defn}
\label{private neighborhood} \emph{A vertex $v$ is called the} private neighbor of $V_i$ with respect to \emph{$f$ if $v\succ V_i$ and $v\nsucc V_j$ for all $j\neq i$. The set of all private neighbors of $V_i$ is called the} private neighborhood \emph{of $V_i$ in $G$ and denoted by $pn_G(V_i;f)$ or simply by $pn(V_i;f)$.}
\end{defn}

The following propositions are useful for our investigations.

\begin{prop}
\label{chi_dt W_n} Let $W_n$ be a wheel of order $n+1\geq 4$. Then
\begin{equation*}
\chi_d^t(W_n)=\left\{
\begin{array}{ll}
3 & \mbox{if }n \mbox{ is even}, \\
4 & \mbox{if }n \mbox{ is odd}.
\end{array}
\right.
\end{equation*}
\end{prop}

\begin{prop}
\label{chi_dt C_n} Let $C_n$ be a cycle of order $n\geq 3$. Then
\begin{equation*}
\chi_d^t(C_n)=\left\{
\begin{array}{ll}
2                                & \mbox{if }n=4, \\
4\lfloor \frac{n}{6}\rfloor +r   & \mbox{if } n\neq 4 \mbox{ and for }r=0,1,2,4,~~n\equiv r\pmod{6},\\
4\lfloor \frac{n}{6}\rfloor +r-1 & \mbox{if }~~n\equiv r\pmod{6} \mbox{, where }r=3,5.
\end{array}
\right.
\end{equation*}
\end{prop}

\begin{proof}
Let $V(C_n)=\{v_i\mbox{ }|\mbox{ }1\leq i \leq n\}$, and let $v_iv_j\in E(C_n)$ if and only if $|i-j|=1$ (to modulo $n$).
We claim that if $f$ is a TDC of $C_n$, then every six consecutive vertices $v_i$, $v_{i+1}$, $v_{i+2}$, $v_{i+3}$, $v_{i+4}$
and $v_{i+5}$ has colored by at least four colors. Trivially, we may assume that some color, say $a$, appear at least two times. We assign colors $a$, $b$, $a$ to vertices $v_i$, $v_{i+1}$, $v_{i+2}$, respectively. We can assign color $b$ to vertex $v_{i+3}$ or not. In each case, we need to at least two new colors $c$ and $d$ for coloring the remained vertices. Because, in the first case, we have to assign two new colors $c$ and $d$ to vertices $v_{i+4}$ and $v_{i+5}$, respectively, and in the second case, we must assign colors $c$, $d$, $c$ to vertices $v_{i+3}$, $v_{i+4}$, $v_{i+5}$, respectively. Therefore, our claim is proved, and we notice that any six consecutive vertices can be colored by four new colors $a$, $b$, $c$, $d$ in
\[
\emph{way 1: a,b,a,b,c,d},~~~ \mbox{   or   } ~~~\emph{way 2: a,b,a,c,d,c}.
\]
We also notice that in \emph{way} 1: $v_{i+1}\in pn(V_a;f)$, $v_{i+2}\in pn(V_b;f)$, $v_{i+3}\in pn(V_c;f)$, $v_{i+4}\in pn(V_d;f)$, and in \emph{way} 2: $v_{i+1}\in pn(V_a;f)$, $v_{i+2}\in pn(V_b;f)$, $v_{i+4}\in pn(V_c;f)$, $v_{i+3}\in pn(V_d;f)$. We continue our proof in the following six cases.

\textbf{Case 0: } $n\equiv 0\pmod{6}$. In this case, if $f_0$ is a proper coloring which is obtained by each of ways 1 or 2 or by combining of them, then $f_0$ will be a TDC of $C_n$ with the minimum number $4\lfloor \frac{n}{6}\rfloor$ color classes, as desired.

\textbf{Case 1: } $n\equiv 1\pmod{6}$. In this case, let $f_0$ be the TDC of $C_n-\{v_n\}$ mentioned in Case 0. Since we need to one new color for coloring $v_n$, by assigning a new color $\varepsilon$ to $v_n$ we obtain a TDC of $C_n$ with the minimum number $4\lfloor \frac{n}{6}\rfloor +1$ color classes, as desired.

\textbf{Case 2: } $n\equiv 2\pmod{6}$. In this case, let $f_0$ be the TDC of $C_n-\{v_{n-1},v_n\}$ mentioned in Case 0. Since we need to two new colors for coloring $v_{n-1}$ and $v_n$, by assigning two new colors $\theta$, $\varepsilon$ to $v_{n-1}$, $v_n$, respectively, we obtain a TDC of $C_n$ with the minimum number $4\lfloor \frac{n}{6}\rfloor +2$ color classes, as desired.

\textbf{Case 3: } $n\equiv 3\pmod{6}$. In this case, let $f_0$ be the TDC of $C_n-\{v_{n-2},v_{n-1},v_n\}$ mentioned in Case 0. Since we need to two new colors for coloring $v_{n-2}$, $v_{n-1}$ and $v_n$, by assigning new colors $\varepsilon$, $\theta$, $\varepsilon$ to $v_{n-2}$, $v_{n-1}$, $v_n$, respectively, we obtain a TDC of $C_n$ with the minimum number $4\lfloor \frac{n}{6}\rfloor +2$ color classes, as desired.

\textbf{Case 4: } $n\equiv 4\pmod{6}$. In this case, let $f_0$ be the TDC of $C_n-\{v_{n-3},v_{n-2},v_{n-1},v_n\}$ mentioned in Case 0. Since we need to four new colors for coloring $v_{n-3}$, $v_{n-2}$, $v_{n-1}$ and $v_n$, by assigning new four colors $\pi$, $\varsigma$, $\theta$, $\varepsilon$ to $v_{n-3}$, $v_{n-2}$, $v_{n-1}$, $v_n$, respectively, we obtain a TDC of $C_n$ with the minimum number $4\lfloor \frac{n}{6}\rfloor +4$ color classes, as desired.

\textbf{Case 5: } $n\equiv 5\pmod{6}$. In this case, let $f_0$ be the TDC of $C_n-\{v_{n-4},v_{n-3},v_{n-2},v_{n-1},v_n\}$ mentioned in Case 0. Since we need to four new colors for coloring $v_{n-4}$, $v_{n-3}$, $v_{n-2}$, $v_{n-1}$, $v_n$, by assigning new colors $\pi$, $\varsigma$, $\pi$, $\theta$, $\varepsilon$ to the vertices $v_{n-4}$, $v_{n-3}$, $v_{n-2}$, $v_{n-1}$ and $v_n$, respectively, we obtain a TDC of $C_n$ with the minimum number $4\lfloor \frac{n}{6}\rfloor +4$ color classes, as desired.
\end{proof}


\begin{prop}
\label{chi_dt P_n} Let $P_n$ be a path of order $n\geq 2$. Then
\begin{equation*}
\chi_d^t(P_n)=\left\{
\begin{array}{ll}
2\lceil \frac{n}{3}\rceil -1   & \mbox{if } n\equiv 1\pmod{3},\\
2\lceil \frac{n}{3}\rceil & \mbox{otherwise}.
\end{array}
\right.
\end{equation*}
\end{prop}

\begin{prop}
\label{chi_d^t Cn Complement} Let $\overline{C_n}$ be the complement of the cycle $C_n$ of order $n\geq 4$. Then
\begin{equation*}
\chi_d^t(\overline{C_n})=\left\{
\begin{array}{ll}
4                                & \mbox{if }n=4,5, \\
\lceil \frac{n}{2}\rceil & \mbox{if }n\geq 6.
\end{array}
\right.
\end{equation*}
\end{prop}

\begin{prop}
\label{chi_d^t Pn Complement} Let $\overline{P_n}$ be the complement of the path $P_n$ of order $n\geq 4$. Then
\begin{equation*}
\chi_d^t(\overline{P_n})=\left\{
\begin{array}{ll}
3                                & \mbox{if }n=4, \\
\lceil \frac{n}{2}\rceil & \mbox{if }n\geq 5.
\end{array}
\right.
\end{equation*}
\end{prop}


\section{\bf Main results}
\vskip 0.4 true cm

The following theorem shows that the total dominator chromatic
number of the Mycieleskian of a graph $G$ lies between $\chi_d^t(G)+1$ and $\chi_d^t(G)+2$.


\begin{thm}
\label{chi_d^t G+1=<chi_d^t M(G)=<chi_d^t G+2} For any graph $G$ with $\delta(G)\geq 1$, $\chi_d ^t(G)+1 \leq \chi_d ^t(M(G)) \leq \chi_d ^t(G)+2$.
\end{thm}

\begin{proof}
Let $f=(V_1,V_2,...,V_{\ell})$ be a $\chi_d^t$-coloring of $G$. Since $g=(V_1,V_2,...,V_{\ell},U,W)$ is a TDC of $M(G)$, we obtain $\chi_d ^t(M(G)) \leq \chi_d ^t(G)+2$. On the other hand, Since $N_{M(G)}(w)=U$, there exists a vertex $u_i\in U$ that is colored by a color different of the colors used in $V(G)$. Hence $\chi_d ^t(M(G)) \geq \chi_d ^t(G)+1$.
\end{proof}

Next results will characterize graphs $G$ satisfying $\chi_d ^t(M(G))=\chi_d ^t(G)+1$ or $\chi_d ^t(M(G))=\chi_d ^t(G)+2$. First a definition.

\begin{defn}
\emph{We say that a graph $G$ belongs to} Class 1 \emph{if it has a $\chi_d^t$-coloring $f=(V_1,V_2,...,V_{\ell})$ such that $pn_G(V_i;f)=\emptyset$ for some $1 \leq i \leq \ell$, and it belongs to} Class 2 \emph{otherwise}.
\end{defn}

The following two theorems present necessary and sufficient conditions for that $\chi_d ^t(M(G))$ be $\chi_d ^t(G)+1$ or $\chi_d ^t(G)+2$.

\begin{thm}
\label{chi M(G)=chi G+1} A graph $G$ with $\delta(G)\geq 1$ belongs to Class 1 if and only if $\chi_d ^t(M(G))=\chi_d ^t(G)+1$.
\end{thm}

\begin{proof}
Let $G$ be a graph in Class 1, and let $f=(V_1,V_2,...,V_{\ell})$ be a $\chi_d^t$-coloring of $G$ such that $pn_G(V_i;f)=\emptyset$ for some $1 \leq i \leq \ell$. Then $g=(V_1, ...,V_{i-1}, V_i\cup\{ w\},V_{i+1},...,V_{\ell},U)$ is a TDC of $M(G)$, and Theorem \ref{chi_d^t G+1=<chi_d^t M(G)=<chi_d^t G+2} implies $\chi_d ^t(M(G))=\chi_d ^t(G)+1$.

Conversely, let $\chi_d ^t(M(G))=\chi_d ^t(G)+1$. By the previous discussion in the proof of Theorem \ref{chi_d^t G+1=<chi_d^t M(G)=<chi_d^t G+2}, we may assume that $f=(V_1,V_2,...,V_{\ell})$ is a $\chi_d^t$-coloring of $M(G)$ such that $w\in V_1$, $V_1-\{w\}\neq \emptyset$ and $V_{\ell}\subseteq U$. Since $w\in V_1$ and there is no edge between $w$ and $V(G)$, we conclude that $v\not\succ V_1$ for any $v\in V(G)$. Therefore the restriction of $f$ on $V(G)$ gives a TDC $g=(V_1-\{w\}-U,V_2-U,...,V_{\ell-1}-U)$ such that $v\not\succ V_1-\{w\}-U$ for any $v\in V(G)$. Hence $pn_G(V_1-\{w\}-U;g)=\emptyset$
and $G$ belongs to Class 1.
\end{proof}

As an immediate consequence, we have the following.

\begin{thm}
\label{chi M(G)=chi G+2} A graph $G$ with $\delta(G)\geq 1$ belongs to Class 2 if and only if $\chi_d ^t(M(G))=\chi_d ^t(G)+2$.
\end{thm}


\section{\bf Characterizing some graphs in Classes 1 and 2}
\vskip 0.4 true cm

Obviously, any complete graph of order at least 3 belongs to Class 1, and any complete $p$-partite graph belongs to Class 2 if and only if $p=2$.
In the next propositions, we show which of wheels, cycles and paths belong to Class 1 and which belong to Class 2.

\begin{prop}
\label{Wn in Class 1} Let $n\geq 3$ be an integer. Every wheel $W_n$ of order $n+1$ belong to in Class 1. Hence
\begin{equation*}
\chi_d^t(M(W_n))=\left\{
\begin{array}{ll}
4 & \mbox{if }n \mbox{ is even}, \\
5 & \mbox{if }n \mbox{ is odd}.
\end{array}
\right.
\end{equation*}.
\end{prop}

\begin{proof}
Let $V(W_n)=\{1,2,...,n+1\}$, for $n\geq 3$, and let vertex $1$ be of degree $n$. If $f=(\{1\},V_2,...,V_{\ell})$ is a $\chi_d ^t$-coloring of $W_n$, then $f'=(\{1\},V_2\cup\{w\},V_3,...,V_{\ell},U)$ is a TDC of $M(W_n)$. Hence
\begin{equation*}
\chi_d ^t(M(W_n))=\chi_d ^t(W_n)+1=\left\{
\begin{array}{ll}
4 & \mbox{if }n \mbox{ is even}, \\
5 & \mbox{if }n \mbox{ is odd},
\end{array}
\right.
\end{equation*}
by Proposition \ref{chi_dt W_n} and Theorem \ref{chi_d^t G+1=<chi_d^t M(G)=<chi_d^t G+2}.
\end{proof}

\begin{prop}
\label{Cn in Class 1} Let $C_n$ be a cycle of order $n\geq 3$. Then $C_n\in Class $ 1 if and only if $n\neq 4,5$ and $n\equiv 4\pmod{6}$.
\end{prop}

\begin{proof}
Let $V(C_n)=\{v_i\mbox{ }|\mbox{ }1\leq i \leq n\}$, and let $v_iv_j\in E(C_n)$ if and only if $|i-j|=1$ (to modulo $n$). First $C_5\in Class $ 1 since $f=(\{v_1,v_3\},\{v_2\},\{v_4\},\{v_5\})$ is a $\chi_d^t$-coloring of $C_5$ and $pn(\{v_2\};f)=\emptyset$. Also $C_4\in Class $ 1 since $C_4$ is the complete bipartite graph $K_{2,2}$. Now let $n=6\ell+4$ for some $\ell \geq 1$. Consider the coloring function $f$ on $V(C_n)$ such that
\begin{equation*}
f(v)=\left\{
\begin{array}{ll}
1+4k & \mbox{if }v=1+6k \mbox{ or } v=3+6k, \\
2+4k & \mbox{if }v=2+6k \mbox{ or } v=4+6k, \\
3+4k & \mbox{if }v=5+6k, \\
4+4k & \mbox{if }v=6+6k, \\
4\ell+m & \mbox{if }v=6\ell +m \mbox{ for } 0< m \leq 4, \\
\end{array}
\right.
\end{equation*}
where $0 \leq k \leq \ell -1$. $f=(V_1,...,V_t)$ is a $\chi_d^t(C_n)$-coloring, where $t=4(\ell+1)$. Then $pn(V_{n-3};f)=\emptyset$ and so $C_n\in Class $ 1. Because $V_{n-3}=\{v_{n-3}\}$, $V_{n-2}=\{v_{n-2}\}$, $V_{n-4}=\{v_{n-4}\}$ and $N(v_{n-3})=\{v_{n-4},v_{n-2}\}$, in which $v_{n-4}\succ V_{n-5}$ and $v_{n-2}\succ V_{n-1}$. In the remained cases we prove $C_n\in Class $ 2. Let $n \equiv 2~\mbox{or}~5\pmod{6}$ and $n\neq 5$. Let also $f=(V_1,...,V_k)$ be an arbitrary $\chi_d^t(C_n)$-coloring. As we saw in the proof of Proposition \ref{chi_dt C_n}, we know that every six consecutive vertices must be colored by one of the ways 1 or 2 or by a combining of them. Also we saw that $pn(V_i;f)\neq \emptyset$, where $1\leq i \leq k-2$. Since $f(v_{n-1})=k-1$, $f(v_{n})=k$, $V_{k-1}=\{v_{n-1}\}$, $V_{k}=\{v_{n}\}$, $pn(V_{k-1};f)=V_{k}$ and $pn(V_{k};f)=V_{k-1}$, we obtain $pn(V_i;f)\neq \emptyset$ for all $i$. Hence $C_n\in Class $ 2. The proof of other cases is similar and we left it to the reader. Therefore we have proved $C_n\in Class $ 1 if and only if $n\neq 4$ and $n\equiv 4\pmod{6}$.
\end{proof}


\begin{prop}
\label{Pn in Class 2} Let $P_n$ be a path of order $n\geq 2$. Then $P_n\in Class $ 2 if and only if $n=2$ or $n\equiv 0\pmod{3}$.
\end{prop}

\begin{proof}
Let $V(P_n)=\{v_i\mbox{ }|\mbox{ }1\leq i \leq n\}$ and for $1\leq i<j\leq n$, $v_iv_j\in E(C_n)$ if and only if $j=i+1$.
Obviously $P_2\in Class $ 2. Now let $n=3\ell+2$ for some $\ell \geq 1$. Consider  a coloring function $f$ defined on $V(P_n)$ such that
\begin{equation*}
f(v_i)=\left\{
\begin{array}{ll}
1+2k & \mbox{if }i=1+3k \mbox{ or } i=3+3k, \\
2+2k & \mbox{if }i=2+3k, \\
\end{array}
\right.
\end{equation*}
when $0 \leq k \leq \ell -2$, and $f(v_{n-4})=f(v_n)=2\ell-1$, $f(v_{n-3})=2\ell$, $f(v_{n-2})=2\ell+1$, $f(v_{n-1})=2\ell+2$. Since $f=(V_1,...,V_{2\ell+2})$ is a $\chi_d^t$-coloring of $P_n$ and $pn(V_{2\ell-1};f)=\emptyset$, we conclude $P_n\in Class $ 1. Now let $n=3\ell+1$ for some $\ell \geq 1$. Consider  a coloring function $f$ defined on $V(P_n)$ such that
\begin{equation*}
f(v_i)=\left\{
\begin{array}{ll}
1+2k & \mbox{if }i=1+3k \mbox{ or } i=3+3k, \\
2+2k & \mbox{if }i=2+3k, \\
\end{array}
\right.
\end{equation*}
when $0 \leq k \leq \ell -2$, and $f(v_{n-3})=f(v_n)=2\ell-1$, $f(v_{n-2})=2\ell$, $f(v_{n-1})=2\ell+1$. Since $f=(V_1,...,V_{2\ell+1})$ is a $\chi_d^t$-coloring of $P_n$ and $pn(V_{2\ell-1};f)=\emptyset$, we conclude $P_n\in Class $ 1. Finally, let $n=3\ell$ for some $\ell \geq 1$. In this case, $P_n$ has the only $\chi_d^t$-coloring $f$ with
\begin{equation*}
f(v_i)=\left\{
\begin{array}{ll}
1+2k & \mbox{if }i=1+3k \mbox{ or } i=3+3k, \\
2+2k & \mbox{if }i=2+3k, \\
\end{array}
\right.
\end{equation*}
when $0 \leq k \leq \ell-1$. Since $pn(V_{i};f)\neq \emptyset$, for all $i$, we conclude $P_n\in Class $ 2.
\end{proof}

\begin{prop}
\label{Cn Complement in Class 1} Let $\overline{C_n}$ be the complement of the cycle $C_n$ of order $n\geq 4$. Then $\overline{C_n}\in Class $ 2 if and only if $n=4,5,6$.
\end{prop}

\begin{proof}
Let $V(\overline{C_n})=\{v_i\mbox{ }|\mbox{ }1\leq i \leq n\}$ and for $1\leq i<j\leq n$, $v_iv_j\in E(\overline{C_n})$ if and only if $j\neq i+1$ (to modulo $n$). If $n=4$, then $C_n$ is isomorphic to the two copies of $K_2$. $K_2\in Class $ 2 implies $\overline{C_4}\in Class $ 2. If $n=5$, then $\overline{C_5}$ is a cycle of order 5, and Proposition \ref{Cn in Class 1} implies $\overline{C_5}\in Class$ 2. In $C_6$, every $\chi_d^t$-coloring is in the form $f=(V_1,V_2,V_3)$, where $V_1=\{v_i,v_{i+1}\}$, $V_2=\{v_{i+2},v_{i+3}\}$, $V_3=\{v_{i+4},v_{i+5}\}$, for some integer $1\leq i \leq 6$. Since $pn(V_i;f)\neq \emptyset$, for all $i$, we obtain $\overline{C_6}\in Class 2$. Now let $n\geq 7$. Assume that $V_i=\{v_{2i},v_{2i-1}\}$, for $1\leq i \leq \lfloor\frac{n}{2}\rfloor$. Since for even $n$, $f=(V_1,V_2,...,V_{\lfloor\frac{n}{2}\rfloor})$ is a $\chi_d^t(\overline{C_n})$-coloring, and for odd $n$, $f=(V_1,V_2,...,V_{\lfloor\frac{n}{2}\rfloor},\{v_n\})$ is a $\chi_d^t(\overline{C_n})$-coloring, and $pn(V_2;f)=\emptyset$, we conclude that $\overline{C_n}\in Class $ 1 if $n\geq 7$.
\end{proof}

\begin{prop}
\label{Pn Complement in Class 1} Let $\overline{P_n}$ be the complement of the path $P_n$ of order $n\geq 4$. Then $\overline{P_n}\in Class $ 1.
\end{prop}

\begin{proof}
Let $V(\overline{P_n})=\{v_i\mbox{ }|\mbox{ }1\leq i \leq n\}$ and for $1\leq i<j<n$, $v_iv_j\in E(\overline{P_n})$ if and only if $j\neq i+1$ (to modulo $n$). If $n=4$, then $\overline{P_4}$ is a path of order 4, and Proposition \ref{Pn in Class 2} implies $\overline{P_4}\in Class$ 1. For $n=5,6$ we consider the $\chi_d^t$-colorings $f=(\{v_1,v_2\},\{v_3,v_4\},\{v_5\})$ and $f=(\{v_1,v_2\},\{v_3,v_4\},\{v_5,v_6\})$, respectively. Since in each case, $pn(\{v_3,v_4\};f)=\emptyset$, we conclude $\overline{P_n}\in Class $ 1. Now let $n\geq 7$. Since the $\chi_d^t(\overline{C_n})$-colorings given in Proposition \ref{Cn Complement in Class 1} are also $\chi_d^t(\overline{P_n})$-colorings and $pn(V_2;f)=\emptyset$, we conclude that $\overline{P_n}\in Class $ 1 if $n\geq 7$.
\end{proof}

By knowing $G\in Class$ 2 if and only if $G\not\in Class$ 1, the previous propositions and Theorems
\ref{chi M(G)=chi G+1} and \ref{chi M(G)=chi G+2} imply the following results.
\begin{prop}
\label{chi_dt M(C_n)} Let $C_n$ be a cycle of order $n\geq 3$. Then
\begin{equation*}
\chi_d^t(M(C_n))=\left\{
\begin{array}{ll}
n                                & \mbox{if }n=4,5 \\
4\lfloor \frac{n}{6}\rfloor +r+2 & \mbox{if }~~n\equiv r\pmod{6} \mbox{, where }r=0,1,2,\\
4\lfloor \frac{n}{6}\rfloor +r+1 & \mbox{if }~~n\equiv r\pmod{6} \mbox{, where }r=3,4,5 \mbox{ and } n\neq 4,5.
\end{array}
\right.
\end{equation*}
\end{prop}

\begin{prop}
\label{chi_dt M(C_n Complement)} Let $C_n$ be a cycle of order $n\geq 3$. Then
\begin{equation*}
\chi_d^t(M(\overline{C_n}))=\left\{
\begin{array}{ll}
6                                & \mbox{if }n=4,5, \\
5                                & \mbox{if }n=6,\\
\lceil\frac{n}{2}\rceil +1 & \mbox{if }n\geq 7.
\end{array}
\right.
\end{equation*}
\end{prop}

\begin{prop}
\label{chi_dt M(P_n)} Let $P_n$ be a path of order $n\geq 2$. Then
\begin{equation*}
\chi_d^t(M(P_n))=\left\{
\begin{array}{ll}
4                                & \mbox{if }n=2, \\
2\lceil \frac{n}{3}\rceil   & \mbox{if } ~~n\equiv 1\pmod{3},\\
2\lceil \frac{n}{3}\rceil +2 & \mbox{if }~~n\equiv 0\pmod{3},\\
2\lceil \frac{n}{3}\rceil +1 & \mbox{if }~~n\equiv 2\pmod{3}, \mbox{ and }n\neq 2.
\end{array}
\right.
\end{equation*}
\end{prop}

\begin{prop}
\label{chi_dt M(P_n Complement)} Let $P_n$ be a path of order $n\geq 2$. Then
\begin{equation*}
\chi_d^t(M(\overline{P_n}))=\left\{
\begin{array}{ll}
4                                & \mbox{if }n=4, \\
\lceil\frac{n}{2}\rceil +1& \mbox{if }n\geq 5.
\end{array}
\right.
\end{equation*}
\end{prop}

We know that for any $k\geq 2$ there is the complete graph $K_k$ with $\chi_d^t(K_k)=k$. Is there a non-complete graph $G$ with $\chi_d^t(G)=k$? Since $\chi_d^t(G)=2$ if and only if $G$ is a complete bipartite graph (see \cite{Kaz}), the answer is negative for $k=2$. In the following we answer this question for $k\geq 3$. First a lemma.

\begin{lem}
\label{G in Class1 then M(G) in Class1} Let $G$ be a graph. If $G \in Class$ 1, then $M(G) \in Class$ 1.
\end{lem}

\begin{proof}
Let $f=(V_1,V_2,...,V_k)$ be a $\chi_d^t$-coloring of $G$ such that $pn_G(V_i;f)=\emptyset$ for some $1\leq i\leq k$. Since $g=(V_1,...,V_{i-1},V_i\cup\{w\},V_{i+1},...,V_k,U)$ is a $\chi_d^t$-coloring of $M(G)$ with this property that $pn_{M(G)}(V_i\cup\{w\};g)=\emptyset$, we conclude $M(G) \in Class$ 1.
\end{proof}

\begin{prop}
\label{For any k>=3 there exists a non-complete graph with chi_d^t=k}
For any $k \geq 3$ there exists a non-complete graph with the total dominator chromatic number $k$.
\end{prop}

\begin{proof}
For $k=3$ there are wheels $W_n\in \mbox{ }Class$ 1 of odd order $n+1\geq 5$ with $\chi_d^t(W_n)=3$, by Propositions \ref{chi_dt W_n} and \ref{Wn in Class 1}. Now let $k\geq 4$. Lemma \ref{G in Class1 then M(G) in Class1} implies that if $G$ belongs to Class 1, then $M^t(G)$ belongs to Class 1 and $\chi_d^t(M^t(G))=\chi_d^t(G)+t$, where $M^t(G)$ is the Mycieleskian of $M^{t-1}(G)$ and $M^0(G)=G$. Now if we consider $G=K_3$ and $t=k-3$, we obtain $\chi_d^t(M^t(G))=k$, and our proof is completed.
\end{proof}

The converse of Lemma \ref{G in Class1 then M(G) in Class1} is false. For example, $K_2\in$ Class 1 but $M(K_2)=C_5\not\in$ Class 1.
Finally, in a natural manner, we end this paper with the following problem.\\

\begin{prob}
For any graph $G$ with no isolated vertex, find some bounds for $\chi_d^t(GM(G))$ in terms of $\chi_d^t(G)$, where $GM(G)$ is the generalized Mycieleskian of $G$.
\end{prob}


\bigskip
\bigskip


{\footnotesize 
{College of Mathematical Sciences}, {University
of Mohaghegh Ardabili, P.O.Box 5619911367,} {Ardabil, Iran}\\
{\tt Email: adelpkazemi@yahoo.com, a.p.kazemi@uma.ac.ir}\\


\begin{thebibliography}{20}

\bibitem{hhs1} T. W. Haynes, S. T. Hedetniemi, P. J. Slater (Eds.),
{\em Fundamentals of Domination in Graphs}, Marcel Dekker, Inc. New
York, 1998.

\bibitem{Kaz} A. P. Kazemi, Total dominator chromatic number in graphs, submitted.

\bibitem{West} D. B. West, \emph{Introduction to Graph Theorey}, 2nd ed., Prentice Hall, USA, 2001.

\end{thebibliography}
\end{document}